\documentclass[11pt]{amsart} 
\usepackage{amssymb,amsmath,latexsym,enumerate,graphicx,bbm,mathptmx,microtype}
\usepackage{hyperref}
\newtheorem{theorem}{Theorem} 
\newtheorem{corollary}[theorem]{Corollary}
\newtheorem{proposition}[theorem]{Proposition}

\newtheorem{exam}{Example}

\newtheorem{rem}{Remark}

\newcommand\commentout[1]{}

\def\Odd{\operatorname{Odd}}
\def\Re{\operatorname{Re}}

\def\Z{\mathbb{Z}}

\def\R{\mathbb{R}}
\def\C{\mathbb{C}}

\def\a{\mathbf{a}}
\def\overrightarrow{\mathbf}
\def\m{\mathbf{m}}

\begin{document}

\title{Relations for Bernoulli--Barnes Numbers and Barnes Zeta Functions}

\dedicatory{Dedicated to the fond memory of Marvin I.\ Knopp (1933--2011)} 

\author{Abdelmejid Bayad}
\address{Laboratoire analyse et probabilit\'es\\
Universit\'e d'Evry Val d'Essonne\\
23 Bd. de France,\\
91037 Evry, France}
\email{abayad@maths.univ-evry.fr}

\author{Matthias Beck}
\address{Department of Mathematics\\
         San Francisco State University\\
         1600 Holloway Ave\\
         San Francisco, CA 94132\\
         U.S.A.}
\email{mattbeck@sfsu.edu}

\begin{abstract}
The \emph{Barnes $\zeta$-function} is
\[
  \zeta_n (z, x; \a) := \sum_{ \m \in \Z_{ \ge 0 }^n } \frac{ 1 }{ \left( x + m_1 a_1 + \dots + m_n a_n \right)^z } 
\]
defined for $\Re(x) > 0$ and $\Re(z) > n$ and continued meromorphically to $\C$.
Specialized at negative integers $-k$, the Barnes $\zeta$-function gives
\[
  \zeta_n (-k, x; \a) = \frac{ (-1)^n k! }{ (k+n)! } \, B_{ k+n } (x; \a)
\]
where $B_k(x; \a)$ is a \emph{Bernoulli--Barnes polynomial}, which can be also defined through a generating function that has a slightly more general form than that for
Bernoulli polynomials.
Specializing $B_k(0; \a)$ gives the \emph{Bernoulli--Barnes numbers}.
We exhibit relations among Barnes $\zeta$-functions, Bernoulli--Barnes numbers and polynomials, which generalize various identities of Agoh, Apostol, Dilcher, and Euler. 
\end{abstract}



\date{6 September 2013}

\thanks{
We thank a referee for helpful comments on a previous version of this paper, in particular, for pointing out Sun's work on self-dual sequences \cite{sundual}, which gave rise to Corollary~\ref{cor:selfdual}.
This research project was initiated while M.\ Beck visited the Universit\'e d'Evry Val d'Essonne; he thanks their Laboratoire analyse et probabilit\'es for their hospitality.
A.\ Bayad was partially supported by the FDIR of the Universit\'e d'Evry Val d'Essonne;
M.\ Beck was partially supported by the US National Science Foundation (DMS-1162638).
}

\maketitle


\section{Introduction}

We define, as usual, the \emph{Bernoulli numbers} $B_k$ through the generating function
\begin{equation}\label{Bernoulli}
  \frac{ z }{ e^z -1 } = \sum_{ k \ge 0 } B_k \, \frac{ z^k }{ k! } \, .
\end{equation}
A fundamental relation of Bernoulli numbers, known at least since Euler's time, is (for $n \ge 1$)
\begin{equation}\label{eulerrelation}
  \sum_{ j=0 }^n \binom n j B_j \, B_{ n-j } = -n \, B_{ n-1 } - (n-1) B_n \, .
\end{equation}
Much more recently, multinomial generalizations of \eqref{eulerrelation} were discovered by Agoh and Dilcher \cite{agohdilcher,dilchersumsproductsbernoulli}. They can be viewed as relations between Bernoulli numbers and \emph{Bernoulli numbers $B_k^{ (n) }$  of order $n$}, defined through
\begin{equation}\label{bernoulli-numbers}
  \left( \frac{ z }{ e^z -1 } \right)^n = \sum_{ k \ge 0 } B_k^{ (n) } \, \frac{ z^k }{ k! } \, .
\end{equation}
Dilcher and others proved generalized  formulas relating $B_k^{ (n)}$ with $B_j$. The first few
are~\cite[p.~32]{dilchersumsproductsbernoulli}:
%
\begin{equation}
\begin{split}
B_k^{(2)}&=-kB_{k-1}-(k-1)B_k\quad \textrm{ (for any } \ k\geq 1), \\
B_k^{(3)}&=k(k-1)B_{k-2}+\tfrac32 k(k-2)B_{k-1}+\tfrac12(k-1)(k-2)B_k\quad \textrm{ (for any } \ k\geq 2), \label{dilcherexamples}\\
B_k^{(4)}&=-k(k-1)(k-2)B_{k-3}+\tfrac{11}6 k(k-1)(k-3)B_{k-2}-k(k-2)(k-3)B_{k-1}-\tfrac16(k-1)(k-2)(k-3)B_{k}\\
&\textrm{ (for any } \ k\geq 3).
\end{split}
\end{equation}
%
Our first goal is to derive relations among \emph{Bernoulli--Barnes numbers} $B_k(\a)$, defined for a fixed vector $\a = \left( a_1, a_2, \dots, a_n \right) \in \R_{ >0 }^n$ through
\begin{equation}\label{Barnes-Bernoulli}
  \frac{ z^n }{ \left( e^{ a_1 z } - 1 \right) \cdots \left( e^{ a_n z } - 1 \right) } = \sum_{ k \ge 0 } B_k(\a) \frac{ z^k }{ k! } \, .
\end{equation}
Note that, with \eqref{bernoulli-numbers} and \eqref{Barnes-Bernoulli}, the Bernoulli--Barnes numbers and Bernoulli numbers are related as 
\[
B_k(\overrightarrow{a})=\sum_{m_1+\dots+m_n=k}\binom{k}{m_1,\dots,m_n}a_1^{m_1-1}\cdots a_n^{m_n-1} B_{m_1}\cdots B_{m_n} \,
.
\]
Of course, one retrieves the Bernoulli numbers of order $n$ with the special case $a_1 = a_2 = \dots = a_n = 1$, and the Bernoulli numbers by further specializing $n=1$.
Our first main result is as follows.

\begin{theorem}\label{firstmainthm}
For $n \ge 3$, $m \ge 1$, where $m$ is odd, and $\a = \left( a_1, a_2, \dots, a_n \right) \in \R_{ >0 }^n$,
\[
  \sum_{ j=n-m }^{ n } \binom{ n+j-4 }{ j-2 } \frac{ 1 }{ (m-n+j)! } \sum_{ |I|=j } B_{ m-n+j } (\a_I) =
  \begin{cases}
     \frac 1 2 & \text{ if } n=m=3, \\
              0 & \text{ otherwise, }
  \end{cases} 
\]
where the inner sum is over all subsets $I \subseteq \left\{ 1, 2, \dots, n \right\}$ of cardinality $j$, and $\a_I := \left( a_i : i \in I \right)$.
\end{theorem}


Here and in what follows below, all binomial coefficients with a negative bottom entry are zero.


\begin{corollary}\label{firstmainCoro*}
For $n \ge 3$ 
and odd $m \ge n-2$,
\[
  \sum_{ j=2 }^n \binom{ n+j-4 }{ j-2 } \frac{ m! }{ (m-n+j)! }\binom{n}{j} 
B_{ m-n+j }^{(j)} = \begin{cases}
     3 & \text{ if } n=m=3, \\
    0 & \text{ otherwise, }
  \end{cases}
\]
\end{corollary}
For example, for $n=3,4$ and odd $m\geq 1$, Theorem \ref{firstmainthm} gives the relations 
\begin{align*}
B_m^{(3)}&=\tfrac32\delta_{3,m}-\tfrac32 m \, B_{m-1}^{(2)} 
\qquad (m\geq 3), \\
B_m^{(4)}&=-2m \, B_{m-1}^{(3)}-m(m-1)B_{m-2}^{(2)} \qquad  (m\geq 4).
\end{align*}
%
More generally, for any positive integer $n\geq 4$, Corollary \ref{firstmainCoro*} gives the following recurrence formula for the numbers $B_m^{(n)}$:
\begin{equation}\label{recurrence1}
\binom{2n-4}{n-2}B_m^{(n)}=-\sum_{ j=2 }^{n-1} \binom{ n+j-4 }{ j-2 } \frac{ m! }{ (m-n+j)! }\binom{n}{j} B_{ m-n+j
}^{(j)}\qquad \textrm{ for any odd } m\geq n-2.
\end{equation}
%
The novelty of these relations as, e.g., compared with \eqref{dilcherexamples} is that they are between Bernoulli numbers of
order higher than 1.
We suspect that there are relations analogous to \eqref{recurrence1} for \emph{even} $m$ but leave the search for them as an
open problem. 

One of the significances of Bernoulli numbers lies in the fact that they are essentially evaluations of the \emph{Riemann $\zeta$-function} $\zeta(z) := \sum_{ m \ge 1 } m^{ -z }$ (meromorphically continued to $\C$) at negative integers $-k$:
\[
  \zeta(-k) = - \frac{ B_{ k+1 } }{ k+1 } \, .
\]
Bernoulli--Barnes numbers appear in a similar fashion in relation with the \emph{Barnes
$\zeta$-function} \cite{barneszeta}
\[
  \zeta_n (z, x; \a) := \sum_{ \m \in \Z_{ \ge 0 }^n } \frac{ 1 }{ \left( x + m_1 a_1 + \dots + m_n a_n \right)^z } 
\]
defined for $\Re(x) > 0$ and $\Re(z) > n$ and continued meromorphically to $\C$ \cite{Elizalde,Ferreira,Friedman-Ruij,Matsumoto,Ruij,Sprea}.
Specialized at negative integers $-k$, the Barnes $\zeta$-function gives
\begin{equation}\label{interpolation}
  \zeta_n (-k, x; \a) = \frac{ (-1)^n k! }{ (m+n)! } \, B_{ k+n } (x; \a)
\end{equation}
where $B_k(x; \a)$ is a \emph{Bernoulli--Barnes polynomial}, defined through \cite{barneszeta}
\begin{eqnarray}\label{Barnes-Bernoulli-poly}
  \frac{ z^n e^{ xz } }{ \left( e^{ a_1 z } - 1 \right) \cdots \left( e^{ a_n z } - 1 \right) } = \sum_{ k \ge 0 } B_k(x; \a) \frac{ z^k }{ k! } \, .
\end{eqnarray}
Thus the Bernoulli--Barnes numbers are the special evaluations $B_k(\a) = B_k(0; \a)$.
It is clear that the Barnes zeta function is a multidimensional generalization of various  Riemann--Hurwitz zetas functions; 
e.g., when $n=1$ and $ \overrightarrow{a}=(a)$, the function $\zeta(s;x,\overrightarrow{a})$ is the classical \emph{Hurwitz zeta function} $a^{-s}\zeta(s; \frac x a ).$ 
Likewise, 
the Bernoulli--Barnes numbers and Bernoulli--Barnes polynomials extend the (generalized) Bernoulli numbers and polynomials to higher dimensions.
Further generalizations of Bernoulli numbers and polynomials include~\cite{kamano,satoh}.

Our second main result expresses the Barnes zeta function in terms of Bernoulli--Barnes polynomials, Hurwitz zeta functions, and \emph{Fourier--Dedekind sums} \cite{bdr}, defined as
  \[ \sigma_{r} \left( a_{ 1 } ,\dots,\widehat{a}_j \dots, a_{ d } ; a_{ j } \right) := \frac{1}{ a_{ j } } \displaystyle\sum_{ \lambda^{ a_{ j } } = 1 \not= \lambda } \frac{ \lambda^{ r } }{ \displaystyle\prod_{1\leq k\neq j\leq n}\left(1-\lambda^{ a_{ k }}\right) } \, . \] 
Fourier--Dedekind sums generalize and unify many variants of (generalized) Dedekind sums; see, e.g., \cite{grosswald} or \cite[Chapter~8]{ccd}.

\begin{theorem}\label{maintheorem}
 Let $a_1,\ldots,a_n$ be pairwise coprime positive integers. Then
\begin{align*}
\zeta(s;x,\overrightarrow{a}) \ =& \
\frac{(-1)^{n-1}}{(n-1)!}\sum_{k=0}^{n-1}(-1)^k\binom{n-1}{k}B_{n-1-k}(x;\overrightarrow{a}) \, \zeta(s-k;x)\\
&\quad + \sum_{j=1}^na_j^{-s}\sum_{r=0}^{a_j-1}\sigma_{-r}(a_1,\dots,\widehat{a}_j,\dots,a_n;a_j)\,
\zeta\left(s;\frac{x+r}{a_j}\right) .
\end{align*}
\end{theorem}

Theorem \ref{maintheorem} has several applications. Specializing $s$ at negative integers gives, with the help of~\eqref{interpolation}:

\begin{corollary}\label{maintheorem2}
 Let $a_1,\ldots,a_n$ be pairwise coprime positive integers. Then
\begin{align*}
&\sum_{j=1}^na_j^{m}\sum_{r=0}^{a_j-1}\sigma_{-r}(a_1,\dots,\widehat{a}_j,\dots,a_n;a_j) \, B_{m+1}\left(\frac{x+r}{a_j}\right)=\\
&\qquad (-1)^{n-1}\frac{(m+1)!}{(m+n)!}B_{m+n}(x,\overrightarrow{a})+\frac{(-1)^{n-1}(m+1)}{(n-1)!}\sum_{k=0}^{n-1}(-1)^k\binom{n-1}{k}B_{n-1-k}(x;\overrightarrow{a})\frac{B_{m+k+1}(x)}{m+k+1}.&
\end{align*}
\end{corollary}

This is reminiscent of a \emph{reciprocity law} for generalized Dedekind sums, due to Apostol \cite{apostoldedekind}; 
this can be illustrated more easily in the case $n=2$ and $\overrightarrow{a}=(a,b)$, for which Theorem \ref{maintheorem} specializes to:

\begin{corollary}\label{coro2}
Let $a, b$ be coprime positive integers. Then
\begin{align*}
  \zeta(s;x,(a,b))
  &=\frac1{ab}\zeta(s-1;x)+ \left( 1- \frac{ x }{ ab } \right) \zeta(s;x) \\
  &\qquad -a^{-s}\sum_{r=0}^{a-1}\left\{\frac{b^{-1}r}a\right\}\zeta\left(s;\frac{x+r}a\right)-b^{-s}\sum_{r=0}^{a-1}\left\{\frac{a^{-1}r}b\right\}\zeta\left(s;\frac{x+r}b\right).
\end{align*}
\end{corollary}

Again the specialization of $s$ at negative integers gives, for $n=2$ and $\overrightarrow{a}=(a,b)$,  using \eqref{interpolation}:

\begin{corollary}
Let $a, b$ be coprime positive integers. Then
\begin{align*}
&a^{m}\sum_{r=0}^{a-1}\left\{\frac{b^{-1}r}a\right\}B_{m+1}\left(\frac{x+r}a\right)+b^{m}\sum_{r=0}^{a-1}\left\{\frac{a^{-1}r}b\right\}B_{m+1}\left(\frac{x+r}b\right
)=\\
&\qquad \frac{1}{m+2}B_{m+2}(x,(a,b))+\frac1{ab}\frac{m+1}{m+2}B_{m+2}(x)+ \left( \frac{ x }{ ab } - 1 \right) B_{m+1}(x) \, .
\end{align*}
\end{corollary}

This is a ``polynomial generalization'' of Apostol's reciprocity law \cite{apostoldedekind} 
\[
  \frac1{m}\left(a^{m-1}s_m(a,b)+b^{m-1}s_m(b,a)\right)=\frac{B_{m+1}}{(m+1)ab}+\frac1{m(m+1)ab}(aB-bB)^{m+1} .
\]
Here $m$ is a positive integer, $a$ and $b$ are coprime, we use the umbral notation
\begin{eqnarray*}
(aB-bB)^{m+1}=\sum_{i=0}^{m+1}\binom{m+1}{i}(-1)^{m+1-i}a^ib^{m+1-i}B_i \, B_{m+1-i} \, ,
\end{eqnarray*}
and
\begin{eqnarray}\label{Apostol}
S_m(a,b)=\sum_{r=0}^{a-1}\left\{\frac{a^{-1}r}b\right\} B_m\left(\frac{r}b\right)
\end{eqnarray}
are the \emph{Apostol--Dedekind sums}.
The classical Dedekind sums \cite{dedekind,grosswald} are captured by the special case $m=1$.
Thus in some sense, our study can be viewed as a bridge between Euler-type identities and Dedekind-type reciprocity laws.




Finally, we discuss the special case $\overrightarrow{a}=(1,\dots,1)$ of Theorem \ref{maintheorem}. Denote  
\[
\zeta_n(s;x)=\zeta(s;x,(1,\dots,1)) \, , 
\] 
the \emph{Hurwitz zeta function of order} $n$. 
Since in this case the sums $\sigma_r(a_1,\dots,\widehat{a}_j,\dots,a_n;a_j)$ vanish, we obtain the following identity.

\begin{corollary}\label{Coro7}
For any positive integer $n$, 
\[
\zeta_n(s;x)=\frac{(-1)^{n-1}}{(n-1)!}\sum_{k=0}^{n-1}(-1)^k\binom{n-1}{k}B^{(n)}_{n-1-k}(x)\, \zeta(s-k;x) \, .
\]
\end{corollary}

Specializing once more $s=-m$ at a negative integer gives: 

\begin{corollary}\label{Coro8}
For any positive integers $n,m$, 
\begin{eqnarray*}
B_{m+n}^{(n)}(x)=(m+n)\binom{m+n-1}{n-1}\sum_{k=0}^{n-1}(-1)^k\binom{n-1}{k}B^{(n)}_{n-1-k}(x)\frac{B_{m+k+1}(x)}{m+k+1}\, .
\end{eqnarray*}
\end{corollary}

Corollary \ref{Coro8} recovers once more Euler's identity \eqref{eulerrelation} and Dilcher's
results in \cite{dilchersumsproductsbernoulli}; it is also reminiscent of the convolution
identities in \cite{pansun}.
We also note that in the above corollaries the coefficients of the polynomials $B^{(n)}_{n-1-k}(x)$ 
can be explicitly given in terms of Stirling numbers of the first kind $s(n,k)$ as follows:
\begin{eqnarray}\label{Stirling}
\binom{n-1}{n-1-k}B^{(n)}_{n-1-k}(x)=\sum_{m=0}^{n-1-k}\binom{m+k}{m} \, s(n,m+k+1) \, x^m
\end{eqnarray}
(see, e.g., \cite[Equation (52.2.21)]{Hansen}). 


\section{Proof of Theorem \ref{firstmainthm}}

Our proof is based on identities of generating functions.
Fix $\a = \left( a_1, a_2, \dots, a_n \right) \in \R_{ >0 }^n$ 
We define for a subset $I \subseteq \left\{ 1, 2, \dots, n \right\}$
\[
  f_I(z) := \frac{ z^{ |I| } e^{ z \sum_{ i \in I } a_i } }{ \prod_{ i \in I } \left( e^{ a_i z } - 1 \right) }
  \qquad \text{ and } \qquad
  F^{ (j) } (z) := \sum_{ |I| = j } f_I (z) \, .
\]

\begin{proposition}\label{mainprop}
For $n \ge 4$,
\[
  \sum_{ j=2 }^n \binom{ n+j-4 }{ j-2 } (-z)^{ n-j } F^{ (j) } (z)
\]
is an even function in~$z$.
\end{proposition}

\begin{proof}
We need to prove that
\[
  F(z) 
  := \sum_{ j=2 }^n \binom{ n+j-4 }{ j-2 } (-z)^{ n-j } F^{ (j) } (z)
  = \sum_{ |I| \ge 2 } \binom{ n+|I|-4 }{ |I|-2 } (-z)^{ n-|I| } f_I(z)
\]
is even, so that if suffices to prove that $F(z)$ equals
\[
  F(-z) = \sum_{ |I| \ge 2 } \binom{ n+|I|-4 }{ |I|-2 } z^{ n-|I| } e^{ - z \sum_{ i \in I } a_i } f_I(z) \, ,
\]
i.e., that
\[
  F(z) - F(-z) = \sum_{ |I| \ge 2 } \binom{ n+|I|-4 }{ |I| - 2 } z^{ n-|I| } f_I(z) \left( (-1)^{ n-|I| } - e^{ - z \sum_{ i \in I } a_i } \right)
\]
is zero. Written with the denominator $\prod_{ i=1 }^n \left( e^{ a_i z } - 1 \right)$, the function $F(z) - F(-z)$ has the numerator
\[
  \sum_{ |I| \ge 2 } \binom{ n+|I|-4 }{ |I| - 2 } z^n e^{ z \sum_{ i \in I } a_i } \left( (-1)^{ n-|I| } - e^{ - z \sum_{ i \in I } a_i } \right) \prod_{ i \notin I } \left( e^{ a_i z } - 1 \right)
\]
and so we can rephrase our goal to proving that
\[
  \sum_{ |I| \ge 2 } \binom{ n+|I|-4 }{ |I| - 2 } \left( (-1)^{ n-|I| } e^{ z \sum_{ i \in I } a_i } - 1 \right) \prod_{ i \notin I } \left( e^{ a_i z } - 1 \right)
\]
is zero. With $\prod_{ i \notin I } \left( e^{ a_i z } - 1 \right) = \sum_{ J \subseteq \overline I } (-1)^{ n-|I|-|J| } e^{ z \sum_{ i \in J } a_i }$, we can further rephrase our goal to proving that
\begin{equation}\label{tobeproved}
  \sum_{ |I| \ge 2 } \binom{ n+|I|-4 }{ |I| - 2 } \sum_{ J \subseteq \overline I } (-1)^{ |J| } e^{ z \sum_{ i \in I \cup J } a_i }
  = \sum_{ |I| \ge 2 } \binom{ n+|I|-4 }{ |I| - 2 } \sum_{ J \subseteq \overline I } (-1)^{ n-|I|-|J| } e^{ z \sum_{ i \in J } a_i }
.
\end{equation}
We will show that the coefficients of $e^{ z \sum_{ i \in K } a_i }$, for any $K \subseteq \left\{ 1, 2, \dots, n \right\}$, on both sides of \eqref{tobeproved} are equal.
This coefficient is on the left-hand side of \eqref{tobeproved} equal to
\[
  \sum_{ j=0 }^{ |K| - 2 } (-1)^j \binom{ |K| }{ j } \binom{ n + |K| - 4 - j }{ |K| - 2 - j } \, .
\]
The corresponding coefficient on the right-hand side of \eqref{tobeproved} is
\[
  \sum_{ I \not\supseteq K } (-1)^{ n-|I|-|K| } \binom{ n + |I| - 4 }{ |I| - 2 }
  = \sum_{ j=2 }^{ n-|K| } (-1)^{ n-j-|K| } \binom{ n + j - 4 }{ j - 2 } \binom{ n-|K| }{ j } \, .
\]
Thus \eqref{tobeproved} is equivalent to
\[
  \sum_{ j=0 }^{ k-2 } (-1)^j \binom{ k }{ j } \binom{ n + k - 4 - j }{ k - 2 - j }
  = \sum_{ j=2 }^{ n-k } (-1)^{ n-j-k } \binom{ n + j - 4 }{ j - 2 } \binom{ n-k }{ j } \, .
\]
But both sides equal $\binom{ n-4 }{ k-2 }$, as one can prove, e.g., by putting either side into a generating function for~$k$.
\end{proof}

\begin{proof}[Proof of Theorem \ref{firstmainthm}]
Recalling that $\a_I = \left( a_i : i \in I \right)$, we see that
\[
  f_I(z)
  = \frac{ z^{ |I| } e^{ z \sum_{ i \in I } a_i } }{ \prod_{ i \in I } \left( e^{ a_i z } - 1 \right) }
  = \frac{ (-z)^{ |I| } }{ \prod_{ i \in I } \left( e^{ -a_i z } - 1 \right) }
\]
is the exponential generating function for $(-1)^k B_k (\a_I)$.
Thus, using Proposition \ref{mainprop} and the notation $\Odd(\dots)$ for the odd part of a function, we compute
\begin{align*}
  0 &= \Odd \left( \sum_{ j=0 }^n \binom{ n+j-4 }{ j-2 } (-z)^{ n-j } F^{ (j) } (z) \right) \\
    &= \Odd \left( \sum_{ j=0 }^n \binom{ n+j-4 }{ j-2 } (-z)^{ n-j } \sum_{ |I|=j } \sum_{ k \ge 0 } (-1)^k B_k (\a_I) \frac{ z^k }{ k! } \right) \\
    &= - \sum_{ j=0 }^n \binom{ n+j-4 }{ j-2 } \sum_{ |I|=j } \sum_{ { k \ge 0 } \atop { n-j+k \text{ odd } } } \frac{ 1 }{ k! } \, B_k (\a_I) \, z^{n-j+k} \\
    &= - \sum_{ { m \ge 0 } \atop { m \text{ odd } } } \sum_{ j=n-m }^{ n } \binom{ n+j-4 }{ j-2 } \frac{ 1 }{ (m-n+j)! } \sum_{ |I|=j } B_{ m-n+j } (\a_I) \, z^m .
\end{align*}
Now read off the coefficients.

In the case $n=3$, a quick calculation reveals
\[
  \Odd \left( \sum_{ j=0 }^3 \binom{ j-1 }{ j-2 } (-z)^{ 3-j } F^{ (j) } (z) \right) = - \frac 1 2 z^3 
\]
and this explains the special case.
\end{proof}

We should remark that Theorem \ref{firstmainthm} was in part motivated by \cite{katayamamultiplebarnes} in which Katayama proposed a three-term generalization of the reciprocity theorem for Dedekind--Apostol sums \cite{apostoldedekind}; Apostol's theorem was a byproduct of another paper of Katayama \cite{katayamabarnes}.
Unfortunately, the main theorem of \cite{katayamamultiplebarnes} is wrong; to make the central integral of the paper work, one has to use the integrand
$
  \frac{ z^{ s-1 } }{ \left( e^{ z/a } - 1 \right) \left( e^{ z/b } - 1 \right) \left( e^{ z/c } - 1 \right) } 
$
which, unfortunately, does not give rise to Dedekind--Apostol sums. However, using this integrand we discovered Theorem~\ref{firstmainthm}.

\section{Proof of Theorem \ref{maintheorem}}



The function 
 \[ p_A (t) := \#\left\{ \left( k_{1} , \dots, k_{n} \right) \in \Z_{ \ge 0 }^{ n } : \, k_{1} a_{1} + \dots + k_{n} a_{n} = t \right\} , \] 
which counts all partitions of $t$ with parts in the finite set 
$ A := \{ a_{ 1 } , \dots, a_{ n } \} $,
is called a \emph{restricted partition function}.
For example, basic combinatorics gives 
\[
  p_{ \{ 1,\dots,1 \} } (t) = \binom{n-1+t}{n-1} \, ,
\]
and a slightly less trivial example was proved by Barlow \cite[p.~323--325]{barlownumbertheory}: for $a$ and $b$ coprime,
\[
  p_{ \{ a , b \} } (t) = \frac{t}{ab}+1-\left\{\frac{b^{-1}t}a\right\}-\left\{\frac{a^{-1}t}b\right\} ,
\]
where $\{ x \} = x - \lfloor x \rfloor$ denotes the fractional part of $x$, $a^{ -1 }$ is computed mod $b$, and $b^{ -1 }$ mod $a$.  
(The above formulation of Barlow's formula seems to be due to Popoviciu \cite{popoviciu}; see also \cite[Chapter~1]{ccd}.)

The following theorem was proved in \cite{bdr}; however, the authors of that paper did not realize the explicit role of Bernoulli--Barnes polynomials.

\begin{theorem}\label{partition}
If $ a_{ 1 } , \dots, a_{ n } $ are pairwise coprime positive integers, then   
  \[ p_A (t) =\frac{(-1)^{n-1}}{(n-1)!} \, B_{n-1} ( -t; (a_{ 1 } , \dots, a_{ n }) ) + \sum_{ j=0 }^{ n } \sigma_{-t} \left( a_{1} , \dots, \widehat{ a_{ j } }, \dots, a_{ n } ; a_{ j } \right) . \]
\end{theorem} 

\begin{proof} 
We give an outline of the proof. As in \cite{bdr}, we compute the residues of  
\[
F_t(z)=\frac{1}{z^{t+1} \prod_{i=1}^n\left(1-z^{a_i}\right)} \, .
\]
The residue at $z=0$ gives $p_A(t)$, whereas the residue at $z=1$ gives
$\frac{(-1)^{n}}{(n-1)!} \, B_{n-1}( -t; (a_{ 1 } , \dots, a_{ n }) )$.
Finally, if $\lambda$ is a nontrivial $a_i$th root of unity,  
\[
\textrm{Res}(F_t(z);\, z=\lambda)=-\frac{1}{a_i \, \lambda^{t} \prod_{j \ne i} \left(1-\lambda^{a_j}\right)} \, ,
\]
where the product runs over all $j = 1, \dots, n$ except $j=i$.
Thus
\[
\sum_{\lambda^{a_i}=1 \ne \lambda} \textrm{Res}(F_t(z); \, z=\lambda) = -\sigma_{-t} \left( a_{1} , \dots, \widehat{ a_{i} }, \dots, a_{ n } ; a_i \right)
.
\]
The residue theorem completes the proof of Theorem~\ref{partition}.
\end{proof}

\begin{proof}[Proof of Theorem \ref{maintheorem}]
Set $t=m_1a_1+\dots+m_na_n.$
We can rewrite the Barnes zeta function as follows: 
\[
\zeta(s;x,\overrightarrow{a})=\sum_{m_1,\ldots,m_n\geq 0}\frac1{(x+m_1a_1+\dots+m_na_n)^s}=\displaystyle\sum_{t\geq 0}\frac{p_A(t)}{(x+t)^s}.
\]
By applying Taylor's theorem to the function $t \mapsto B_{n-1}(-t;\overrightarrow{a})$ at $t=-x$, 
\[
B_{n-1}(-t;\overrightarrow{a})=\sum_{k=0}^{n-1}(-1)^k\binom{n-1}{k}B_{n-1-k}(x;\overrightarrow{a})(x+t)^k ,
\]
and so with Theorem \ref{partition} we obtain 
\[
p_A(t)= \frac{(-1)^{n-1}}{(n-1)!}\sum_{k=0}^{n-1}(-1)^k\binom{n-1}{k}B_{n-1-k}(x;\overrightarrow{a})(x+t)^k + \sum_{ j=0 }^{ n } \sigma_{-t} \left( a_{1} , \dots, \widehat{ a_{ j } }, \dots, a_{ n } ; a_{ j } \right).
\]
Hence
\[
\zeta(s;x,\overrightarrow{a})=\frac{(-1)^{n-1}}{(n-1)!}\sum_{k=0}^{n-1}(-1)^k\binom{n-1}{k}B_{n-1-k}(x;\overrightarrow{a})\,
\zeta(s-k;x)+\sum_{j=1}^n\sum_{t\geq
0} \frac{ \sigma_{-t}(a_1,\dots,\widehat{a}_j,\dots,a_n;a_j) }{ (x+t)^s} \, .
\]
Note that $\sigma_{-t} \left( a_{1} , \dots, \widehat{ a_{ j } }, \dots, a_{ n } ; a_{ j } \right)$ depends only $t\bmod{a_j}$. Setting $t=ma_j+r$, $0\leq r\leq a_j-1$, we obtain
\begin{align*}
  \zeta(s;x,\overrightarrow{a}) &=
\frac{(-1)^{n-1}}{(n-1)!}\sum_{k=0}^{n-1}(-1)^k\binom{n-1}{k}B_{n-1-k}(x;\overrightarrow{a})\, \zeta(s-k;x) \\
  &\qquad + \sum_{j=1}^n\sum_{r=0}^{a_j-1}\sigma_{-r}(a_1,\dots,\widehat{a}_j,\dots,a_n;a_j)\sum_{m\geq 0}\frac1{(x+r+ma_j)^s} \, .
\end{align*}
Writing $(x+r+ma_j)^{-s}=a_j^{-s}\left(\frac{x+r}{a_j}+m\right)^{-s}$ completes the proof of Theorem~\ref{maintheorem}.
\end{proof}

\section{The special case $\overrightarrow{a} = (a, 1, 1, \dots, 1)$}

In the special case $A = \left\{ a, 1, 1, \dots, 1 \right\}$ (with $n$ 1's), most of the terms in
Theorem \ref{partition} disappear and we obtain 
\begin{equation}\label{specialpartition}
  p_{ \left\{ a, 1, 1, \dots, 1 \right\} } (t) = \frac{(-1)^{n}}{n!} \, B_{n}(-t; (a, 1, 1, \dots, 1)) + \sigma_{-t} \left( 1, 1, \dots, 1 ; a \right) . 
\end{equation}
On the other hand, we can apply \cite[Theorem 8.8]{ccd} to this special case; thus for $t=1, 2, \dots, a+n-1$,
\[
  \sigma_{t} \left( 1, 1, \dots, 1 ; a \right) 
  = \frac{(-1)^{n-1}}{n!} \, B_{n}(t;(a, 1, 1, \dots, 1)) \, .
\]
Since $\sigma_{t} \left( 1, 1, \dots, 1 ; a \right)$ only depends on $t \bmod a$, this range for $t$ is enough to determine $\sigma_{t} \left( 1, 1, \dots, 1 ; a \right)$:
\begin{eqnarray}\label{explicit-Dedekind-sums}
  \sigma_{t} \left( 1, 1, \dots, 1 ; a \right) 
  = \begin{cases}
    \frac{(-1)^{n-1}}{n!} \, B_{n}(t \bmod a; \, (a, 1, 1, \dots, 1)) & \text{ if } t \not\equiv 0 \bmod a, \\
    \frac{(-1)^{n-1}}{n!} \, B_{n}(a; \, (a, 1, 1, \dots, 1)) & \text{ if } t \equiv 0
\bmod a.
  \end{cases}
\end{eqnarray}
For the case $t \equiv 0 \bmod a$ we can also use \cite[Theorem 8.4]{ccd} which gives
\[
  \sigma_{0} \left( 1, 1, \dots, 1 ; a \right) 
  = 1 - \frac{(-1)^{n}}{n!} \, B_{n}(0; (a, 1, 1, \dots, 1)) \, .
\]
(An easy way to see that our two formulations of $\sigma_{0} \left( 1, 1, \dots, 1 ; a \right)$ are
equivalent is through the difference formula
\[
  B_m \left( x+a_0; (a_0, a_1, \dots a_n) \right) - B_m \left( x; (a_0, a_1, \dots a_n) \right) = m \, B_{
m-1 } \left( x; (a_1, a_2, \dots a_n) \right)
\]
and then specializing this to $m=n$, $x=0$, $a_0 = a$, and $a_1 = a_2 = \dots = a_n = 1$.) 
Substituting this back into \eqref{specialpartition} gives, with $\chi_a(t) := 1$ if $a|t$ and $\chi_a(t) := 0$ otherwise:

\begin{proposition}
$
  p_{ \left\{ a, 1, 1, \dots, 1 \right\} } (t) = 
  \frac{(-1)^{n}}{n!} \left( B_{n}(-t; (a, 1, 1, \dots, 1)) - B_{n}(t \bmod a; \, (a, 1, 1, \dots,
1))  \right) + \chi_a(t) \, .
$
\end{proposition}

Using \eqref{explicit-Dedekind-sums} and Theorem \ref{maintheorem} we obtain: 
\begin{proposition}\label{mainproposition} 
Let $\overrightarrow{a} = (a, 1, 1, \dots, 1)$, where $a$ is a positive integer. Then
\[
\zeta(s;x,\overrightarrow{a}) = 
\frac{(-1)^{n-1}}{(n-1)!}\displaystyle\sum_{k=0}^{n-1}(-1)^k\binom{n-1}{k}B_{n-1-k}(x;\overrightarrow{a}) \, \zeta(s-k;x)+\frac{(-1)^{n-1}}{n!} a^{-s}\sum_{r=1}^{a}B_{n}(r;\overrightarrow{a}) 
\, \zeta\left(s;1+\frac{x-r}{a}\right).
\]
\end{proposition}

Specializing $s=-m$ at negative integers gives, by Proposition \ref{mainproposition} with the help of \eqref{interpolation}, the following formula.

\begin{corollary}
Let $\overrightarrow{a} = (a, 1, 1, \dots, 1)$, where $a$ is a positive integer. Then
\[
\frac{m!n!}{(m+n)!}B_{m+n}(x,\overrightarrow{a})=
\sum_{k=0}^{n-1}(-1)^k(k+1)\binom{n}{k+1}B_{n-1-k}(x;\overrightarrow{a})\frac{B_{m+k+1}(x)}{m+k+1}+
a^{m}\sum_{r=1}^{a}B_{n}(r;\overrightarrow{a})\frac{B_{m+1}\left(1+\frac{x-r}{a}\right)}{m+1} \, .
\]
\end{corollary}

\section{Difference, symmetry and recurrence formulas for $B_{m}(x;\overrightarrow{a})$}


We conclude by giving various formulas for $B_{m}(x;\overrightarrow{a})$, starting with the following difference formula. 
\begin{theorem} 
For $\overrightarrow{a}=(a_1,\dots, a_n)\in \mathbb R_ {\geq 0}^n$, we have the difference formula
\[
(-1)^mB_m(-x;\overrightarrow{a})-B_m(x;\overrightarrow{a})=m! \displaystyle\sum_{k=0}^{n-1}\displaystyle\sum_{|I|=k}\frac{B_{m-n+k}(x;\overrightarrow{a}_I)}{(m-n+k)!}\ , \
\]
with $B_m(x,\overrightarrow{a}_I)=x^m$ if $I=\emptyset$. Furthermore,
\[
B_m(x+\sum_{i=1}^na_i;\overrightarrow{a})=(-1)^mB_m(-x;\overrightarrow{a}) \, .
\]
\end{theorem}
\begin{proof}
By use of the identity
\[
\sum_{I\subset \{1,\dots,n\}}\prod_{i\in\bar{I}}\left(e^{a_i t}-1\right)=e^{(a_1+\cdots+a_n)t},
\]
we obtain the formula 
\begin{eqnarray}\label{(16)}
\displaystyle\frac{t^ne^{(x+\sum_{i=1}^na_i)t}}{\displaystyle\prod_{i=1}^n\left(e^{a_it}-1\right)}=\displaystyle\sum_{I\subset
\{1,\dots,n\}}\frac{t^ne^{xt}}{\displaystyle\prod_{i\in I}\left(e^{a_it}-1\right)} \, ,
\end{eqnarray}
where $\bar{I}=\{1,\dots,n\}\backslash I.$ On the other hand, we have the equality
\begin{eqnarray}\label{(17)}
\displaystyle\frac{t^ne^{(x+\sum_{i=1}^na_i)t}}{\displaystyle\prod_{i=1}^n\left(e^{a_it}-1\right)}=\displaystyle\frac{(-t)^ne^{xt}}{\displaystyle\prod_{i=1}^n\left(e^{-a_it}-1\right)}
\, .
\end{eqnarray}
Therefore, by \eqref{(16)} and \eqref{(17)},
\[
\displaystyle\sum_{I\subset \{1,\dots,n\}}\frac{t^ne^{xt}}{\displaystyle\prod_{i\in
I}\left(e^{a_it}-1\right)}=\displaystyle\frac{(-t)^ne^{xt}}{\displaystyle\prod_{i=1}^n\left(e^{-a_it}-1\right)} \, .
\]
Together with \eqref{Barnes-Bernoulli-poly} this completes the proof.
\end{proof}

Our next result is a symmetry formula.
\begin{theorem}\label{Symetry} 
Let $\overrightarrow{a}=(a_1,\dots, a_n)\in \mathbb R_ {\geq 0}^n$ with $A:= a_1+\cdots+a_n >0$. 
Then for any integers $l,m\geq 0$,
\begin{equation}\label{symm1}
(-1)^m\displaystyle\sum_{k=0}^{m}\binom{m}{k}A^{m-k}B_{l+k}(x;\overrightarrow{a})= (-1)^l\displaystyle\sum_{k=0}^{l}\binom{l}{k}A^{l-k}B_{m+k}(-x;\overrightarrow{a}),
\end{equation}
and 
\begin{eqnarray}\label{symm2}
\begin{split}
&\frac{(-1)^m}{m+l+2}\ \displaystyle\sum_{k=0}^{m}\binom{m+1}{k}(l+k+1) \, A^{m+1-k} \, B_{l+k}(x;\overrightarrow{a})&\\
&\qquad + \frac{(-1)^l}{m+l+2}\ \displaystyle\sum_{k=0}^{n}\binom{l+1}{k}(l+k+1) \, A^{l+1-k}\, B_{m+k}(-x;\overrightarrow{a})&\\
&=
(-1)^{m+1} B_{l+m+1}(x;\overrightarrow{a})+(-1)^{l+1}B_{n+m+1}(-x;\overrightarrow{a}) \, . 
\end{split}
\end{eqnarray}
\end{theorem}
\begin{proof} Observe that from \eqref{Barnes-Bernoulli-poly} we obtain 
\[
\frac{d}{dx}B_{l+1}(x;\overrightarrow{a})=(l+1)B_{l}(x;\overrightarrow{a})
\]
and by applying the operator $\frac{d}{dx}$ to \eqref{symm1}, this implies  \eqref{symm2}.

Now we prove \eqref{symm1}. Consider the generating function
\begin{align*}
&\sum_{m\geq 0}\sum_{l\geq 0}\left( (-1)^m\sum_{k=0}^{m}\binom{m}{k}A^{m-k}B_{l+k}(x;\overrightarrow{a})\right)\frac{y^m}{m!} \frac{z^l}{l!} \\
&= \sum_{l\geq 0}\sum_{k\geq 0}A^{-k}B_{l+k}(x;\overrightarrow{a})  \frac{z^l}{l!}\sum_{m=k}^{\infty}(-1)^m\binom{m}{k}\frac{(Ay)^m}{m!} \\
&= \sum_{l\geq 0}\sum_{k\geq 0}B_{l+k}(x;\overrightarrow{a})  \frac{z^l}{l!}\frac{(-y)^k}{k!} e^{-Ay}=e^{-Ay}\sum_{i\geq 0}\sum_{k=0}^iB_{i}(x;\overrightarrow{a})  \frac{z^{i-k}}{(i-k)!}\frac{(-y)^k}{k!} \\
&= \sum_{l\geq 0}\sum_{k\geq 0}B_{l+k}(x;\overrightarrow{a})  \frac{z^l}{l!}\frac{(-y)^k}{k!} e^{-Ay}=e^{-Ay}\sum_{i\geq 0}\sum_{k=0}^i\frac{B_{i}(x;\overrightarrow{a})}{i!}\binom{i}{k}z^{i-k}(-y)^k\\
&= e^{-Ay}\sum_{i\geq 0}B_{i}(x;\overrightarrow{a})\frac{(z-y)^i}{i!}=e^{-Ay}\frac{(z-y)^n
e^{x(z-y)}}{\prod_{i=1}^n\left(e^{a_i(z-y)}-1\right)}=\frac{(z-y)^n
e^{x(z-y)}}{\prod_{i=1}^n\left(e^{a_iz}-e^{a_iy}\right)} \, .
\end{align*}
Similarly the generating function is also equal to
\[
\sum_{m\geq 0}\sum_{l\geq 0}\left(
(-1)^l\displaystyle\sum_{k=0}^{l}\binom{l}{k}A^{l-k}B_{m+k}(-x;\overrightarrow{a})\right)\frac{y^m}{m!} \frac{z^l}{l!}
=\frac{(z-y)^n e^{x(z-y)}}{\displaystyle\prod_{i=1}^n\left(e^{a_iz}-e^{a_iy}\right)} \, . \qedhere
\]
\end{proof}
Specializing  Theorem \ref{Symetry} at $l=m$ we obtain a recurrence formula for the polynomials $$P_m(x):= (m+1) \, A^{-m}\Big(B_{m}(-x;\overrightarrow{a})+B_{m}(x;\overrightarrow{a})\Big)$$ as follows.
\begin{corollary} 
For any positive integer $m\geq 1$,
\begin{eqnarray*}
P_{2m+1}(x)=- \frac{2m+1}{2(m+1)}\displaystyle\sum_{k=0}^{m}\binom{m+1}{k}P_{m+k}(x) \, . 
\end{eqnarray*}
In the case $n=1$, $a_1=1$, the polynomials $P_m(x)$ are reduced to $(m+1)\Big(B_m(-x)+B_m(x)\Big).$
\end{corollary}

The form of \eqref{symm1} is reminiscent of identities for self-dual sequences, in particular,
\cite{sundual}. Namely, setting
\[
  Q_k(x; \a) := (-1)^k A^{ -k } B_k(x; \a) \, ,
\]
the case $l=0$ in Theorem \ref{Symetry} gives
\[
  Q_m(-x; \a) = \sum_{ k=0 }^m \binom m k (-1)^k Q_k(x; \a) =: Q_k^*(x; \a) \, ,
\]
the \emph{dual} of $Q_k(x; \a)$ (viewed as a sequence with index $k$).
In particular, the sequence $Q_k(0; \a)$ is \emph{self dual}:
\[
  Q_m(0; \a) = Q_k^*(0; \a) \, ,
\]
Let's record this:

\begin{corollary}\label{cor:selfdual}
Let $\overrightarrow{a}=(a_1,\dots, a_n)\in \mathbb R_ {\geq 0}^n$ with $A:= a_1+\cdots+a_n >0$.
Then $\bigl( (-1)^n A^{ -n } B_n(\a) \bigr)_{ n \in \Z_{ \ge 0 } }$ is a self-dual sequence.
\end{corollary}

It would be interesting to prove this corollary directly (i.e., without referring to Theorem
\ref{Symetry}).
In fact, if this is possible, one could then apply \cite[Theorem 1.1]{sundual} to give an
independent proof of Theorem~\ref{Symetry}.

As a final result, we obtain the following recurrence formula for the Bernoulli--Barnes numbers.
\begin{theorem} 
For any positive integer $m\geq 1$,
\begin{eqnarray*}
B_{2m+1}(\overrightarrow{a})=-\frac{1}{2(m+1)} \displaystyle\sum_{k=0}^{m}\binom{m+1}{k}(m+k+1) \,
A^{m+1-k}B_{m+k}(\overrightarrow{a})
\end{eqnarray*}
and
\begin{eqnarray*}
\begin{split}
B_{2m}(\overrightarrow{a})=& -\frac1{(m+1)(2m+1)}\displaystyle\sum_{k=0}^{m}\binom{m+1}{k}(m+k+1) \, A^{m-k}B_{m+k}(\overrightarrow{a})&\\
&-(2m)!\ A^{-1} \displaystyle\sum_{k=0}^{n-1}\displaystyle\sum_{|I|=k}\frac{B_{2m+1-n+k}(\overrightarrow{a}_I)}{(2m+1-n+k)!} \, .&
\end{split}
\end{eqnarray*}
\end{theorem}

Note that for $n=1$, the above results specialize to the well-known difference, symmetry and recurrence concerning the ordinary Bernoulli numbers and polynomials. 


\bibliographystyle{amsplain}
\bibliography{bib}

\setlength{\parskip}{0cm} 

\end{document}